\documentclass[12pt]{article}
\usepackage[font=small]{caption}
\usepackage{tikz}
\usepackage[T1]{fontenc}
\usepackage{amssymb}
\usepackage{amsmath}
\usepackage{amsfonts}
\usepackage{amsthm}
\usepackage{hyperref}
\usepackage{enumitem}
\usepackage[showonlyrefs]{mathtools}
\mathtoolsset{showonlyrefs}
\numberwithin{equation}{section}
\newtheorem{theorem}{Theorem}[section]
\newtheorem{thmx}{Theorem}

\newtheorem*{theorembb}{Theorem B$^\prime$}
\newtheorem{lemma}{Lemma}[section]
\newtheorem{proposition}{Proposition}[section]

\theoremstyle{remark}

\newtheorem*{ack}{Acknowledgment}
\def\F{\mathcal{F}}
\def\R{\mathbb{R}}

\def\C{\mathbb{C}}

\def\N{\mathbb{N}}
\def\D{\mathbb{D}}
\def\O{\mathcal{O}}
\def\re{\operatorname{Re}}
\def\im{\operatorname{Im}}

\begin{document}
\title{Radially distributed values and normal families, II}
\author{Walter Bergweiler and Alexandre Eremenko\thanks{Supported by NSF grant DMS-1665115.}}
\date{}
\maketitle
\begin{abstract}
We consider the family of all functions holomorphic in the unit disk
for which the zeros lie on one ray while the $1$-points lie on two different 
rays. We prove that for certain configurations of the rays this family is normal outside the
origin.
\end{abstract}
\section{Introduction and results}
There is an extensive literature on entire functions whose zeros and $1$-points are distributed
on finitely many rays. One of the first results of this type is the following theorem of 
Biernacki \cite[p.~533]{Biernacki1929} and Milloux~\cite{Milloux1927}.
\begin{thmx} \label{thm-bier-mill}
There is no transcendental entire function for which all zeros lie on one ray and
all $1$-points lie on a different ray.
\end{thmx}
Biernacki and Milloux proved this under the additional hypothesis that the function considered has finite 
order, but by a later result of  Edrei~\cite{Edrei1955} this is always the case 
if all zeros and $1$-points lie on finitely many rays.

A thorough discussion of the cases in which an entire function can have its zeros on one system 
of rays and its $1$-points on another system of rays, intersecting the first one only at $0$,
 was given in~\cite{BEH}. Special attention
was paid to the case where the zeros are on one ray while the $1$-points are on two rays. 
For this case the following result was obtained \cite[Theorem~2]{BEH}.
\begin{thmx} \label{thm-BEH}
Let $f$ be a transcendental entire function whose zeros lie on a ray $L_0$ and whose $1$-points
lie on two rays $L_1$ and $L_{-1}$, each of which is distinct from $L_0$. Suppose that
the numbers of zeros and $1$-points are infinite.
Then $\angle(L_0,L_1)=\angle(L_0,L_{-1})<\pi/2$.
\end{thmx}
The hypothesis that $f$ has infinitely many zeros excludes the example
$f(z)=e^z$ in which case we have $\angle(L_1,L_{-1})=\pi$, and $L_0$ can be taken 
arbitrarily.
Without this hypothesis we have the following result.
\begin{theorembb} 
Let $f$ be a transcendental entire function whose zeros lie on a ray $L_0$ and whose $1$-points
lie on two rays $L_1$ and $L_{-1}$, each of which is distinct from $L_0$. 
Then $\angle(L_1,L_{-1})=\pi$ or $\angle(L_0,L_1)=\angle(L_0,L_{-1})<\pi/2$.
\end{theorembb}

Bloch's heuristic principle says that the family of all functions holomorphic in some domain which have a
certain property is likely to be normal if there does not exist a non-constant entire function with
this property. More generally, properties which are satisfied only by ``few'' entire functions often
lead to normality. 
We refer to~\cite{Bergweiler2006}, \cite{Steinmetz2017} and~\cite{Zalcman1998}
for a thorough discussion of Bloch's principle.

The following normal family analogue of Theorem~\ref{thm-bier-mill} was proved in~\cite[Theorem~1.1]{BE}.
Here $\D$ denotes the unit disk.
\begin{thmx} \label{thm-BE}
Let $L_0$ and $L_1$ be two distinct rays emanating from the origin and
let $\F$ be the family of all functions holomorphic in $\D$ for which all zeros
lie on $L_0$ and all $1$-points lie on $L_1$.
Then $\F$ is normal in $\D\backslash\{0\}$.
\end{thmx}
The purpose of this paper is to prove a normal family analogue of Theorem~B$^\prime$.
\begin{theorem}
\label{thm1}
Let $L_0$, $L_1$ and $L_{-1}$ be three distinct rays emanating from the origin and
let $\F$ be the family of all functions holomorphic in $\D$ for which all zeros lie
on $L_0$ and all $1$-points lie on $L_{1}\cup L_{-1}$.
Assume that neither $\angle(L_{-1},L_1)=\pi$ nor
$\angle(L_0,L_1)=\angle(L_0,L_{-1})< \pi/2$.
Then $\F$ is normal  in $\D\setminus \{0\}$.
\end{theorem}
It was shown in~\cite[Theorem~3]{BEH} that if $\alpha$ is of the form $\alpha=2\pi/n$ 
with $n\in\N$, $n\geq 5$, then there exist rays $L_0$ and $L_{\pm 1}$ with
$\angle(L_0,L_1)=\angle(L_0,L_{-1})=\alpha$ and an entire function $f$ with all 
zeros on $L_0$ and all $1$-points on $L_1$ and $L_{-1}$.
In~\cite{Eremenko2015} such an entire function $f$ was constructed for every $\alpha\in (0,\pi/3]$.

The functions constructed in~\cite{BEH,Eremenko2015} have the property that $f(re^{i\theta})\to 0$ 
as $r\to\infty$ for $|\theta|<\alpha$ while
$f(re^{i\theta})\to \infty$ as $r\to\infty$ for $\alpha<|\theta|\leq\pi$.
Considering the family $\{f(kz)\}_{k\in\N}$ we see that the conclusion of Theorem~\ref{thm1}
does not hold if $\angle(L_0,L_1)=\angle(L_0,L_{-1})\in (0,\pi/3]\cup\{2\pi/5\}$.
The example $\{e^{kz}\}_{k\in\N}$ shows that it does not hold if $\angle(L_{-1},L_1)=\pi$.

The question whether the conclusion of Theorem~\ref{thm1} holds if
$\angle(L_0,L_1)=\angle(L_0,L_{-1})\in (\pi/3,\pi/2) \setminus\{2\pi/5\}$ remains open.

We note that Theorem~B$^\prime$ follows from Theorem~\ref{thm1}. To see this we only
have to note that if $f$ is a transcendental entire function and $(z_k)$ is a sequence
tending to $\infty$ such that $|f(z_k)|\leq 1$ for all $k\in\N$, then 
$\{f(2|z_k|z)\}_{k\in\N}$ is not normal at some point of modulus~$\frac12$; see the remark
after Theorem~1.1 in~\cite{BEH}.

A key tool in the theory of normal families is Zalcman's lemma~\cite{Zalcman1975}; see Lemma~\ref{lemma1} below.
An extension of this result (Lemma~\ref{lemma2} below)
was also crucial in the proof of Theorem~\ref{thm-BE} in~\cite{BE}.
In fact, this extension was  used to prove the
following result~\cite[Theorem 1.3]{BE} from which Theorem~\ref{thm-BE} can be deduced.
\begin{thmx} \label{thm-BE1}
Let $D$ be a domain and let $L$ be a straight line which divides
$D$ into two subdomains $D^+$ and $D^-$.
Let $\F$ be a family of functions holomorphic in $D$ which do not have zeros
in $D$ and for which all $1$-points lie on~$L$.

Suppose that $\F$ is not normal at $z_0\in D\cap L$ and let $(f_k)$ be a sequence
in $\F$ which does not have a subsequence converging in any neighborhood of~$z_0$.
Suppose that $(f_k|_{D^+})$ converges.
Then either $f_k|_{D^+}\to 0$ and $f_k|_{D^-}\to \infty$  or
$f_k|_{D^+}\to \infty$ and $f_k|_{D^-}\to 0$.
\end{thmx}
Note that $\F$ is normal in $D^+$ by Montel's theorem.
So it is no restriction to assume that $(f_k|_{D^+})$ converges, since this can be achieved
by passing to a subsequence.

Theorem~\ref{thm-BE1} will also play an important role in the proof of Theorem~\ref{thm1}.
However, we will also need the following addendum to Theorem~\ref{thm-BE1}.
Here and in the following $D(a,r)$ and $\overline{D}(a,r)$ denote
the open and closed disk of radius $r$ centered at a point $a\in\C$.

\begin{proposition}\label{prop1}
Let $D$, $L$, $\F$, $z_0$ and $(f_k)$ be as in Theorem~\ref{thm-BE1}.
Let $r>0$ with $\overline{D}(z_0,r)\subset D$.
Then for sufficiently large $k$ there exists a $1$-point $a_k$ of $f_k$ such that $a_k\to z_0$ and
if $M_k$ is the line orthogonal to $L$ which intersects $L$ at $a_k$, then
$|f_k(z)|\neq 1$ for $z\in M_k\cap \overline{D}(z_0,r)\setminus \{a_k\}$.
\end{proposition}
For large $k$ this yields that $|f_k(z)|>1$ for $z\in M_k\cap D^+\cap \overline{D}(z_0,r)$ and
$|f_k(z)|<1$ for $z\in M_k\cap D^-\cap \overline{D}(z_0,r)$, or vice versa.

\begin{ack}
We thank the referee for helpful comments.
\end{ack}
\section{Preliminaries}
The lemma of Zalcman already mentioned in the introduction is the following.
\begin{lemma}\label{lemma1}
{\bf (Zalcman's Lemma)}
Let $\mathcal{F}$ be a family of functions meromorphic in a domain $D$ in $\C$.
Then $\F$ is not normal at a point $z_0\in D$ if and only if there exist
\begin{enumerate}[label=$(\roman*)$,itemsep=0pt, topsep=5pt]
\item points $z_k\in D$ with $z_k\to z_0$,
\item positive numbers $\rho_k$ with $\rho_k\to 0$,
\item functions $f_k\in \F$
\end{enumerate}
such that
\begin{equation}\label{1a}
f_k(z_k+\varrho_kz)\to g(z)
\end{equation}
locally uniformly in $\C$ with respect to the spherical metric, 
where $g$ is a non-constant meromorphic function in~$\C$.
\end{lemma}
In the proof (see also~\cite[Section~4]{Bergweiler1998}
or~\cite[p.~217f]{Zalcman1998} besides~\cite{Zalcman1975})
one considers the spherical derivative 
\begin{equation}\label{1a1}
g_k^\#(z)=\frac{|g_k'(z)|}{1+|g_k(z)|^2}
\end{equation}
of the function $g_k$ defined by
\begin{equation}\label{1b}
g_k(z)=f_k(z_k+\varrho_kz)
\end{equation}
and shows that for suitably chosen $f_k$, $z_k$, $\varrho_k$  and $R_k$
with $R_k\to\infty$ we have $g_k^\#(0)=1$ as well as
\begin{equation}\label{1c}
g_k^\#(z)\leq 1+o(1)
\quad 
\text{for }|z|\leq R_k
\text{ as }k\to\infty.
\end{equation}
Marty's theorem then implies that $(g_k)$ has a locally convergent subsequence.

The following addendum to Lemma~\ref{lemma1} was proved in~\cite[Lemma~2.2]{BE}.
\begin{lemma}\label{lemma2}
Let $t_0>0$ and $\varphi\colon [t_0,\infty)\to (0,\infty)$ be a non-decreasing function such that
$\varphi(t)/t\to 0$ as $t\to\infty$ and
\begin{equation}\label{1d}
\int^\infty_{t_0}\frac{dt}{t\varphi(t)}<\infty.
\end{equation}
Then one may choose $z_k$, $\varrho_k$ and $f_k$ in Zalcman's Lemma~\ref{lemma1} such that
\begin{equation}\label{1d1}
R_k:=\frac 1 {\varrho_k\varphi\!\left(1/\varrho_k\right)}\to\infty
\quad \text{as }k\to\infty.
\end{equation}
and the functions $g_k$ given by \eqref{1b} are defined in the disks $D(0,R_k)$
and satisfy
\begin{equation}\label{1d2}
g_k^\#(z)\leq 1+\frac{|z|}{R_k}\quad\text{for }|z|<R_k.
\end{equation}
\end{lemma}

The next lemma is standard~\cite[Proposition~1.10]{Pommerenke1992}.
\begin{lemma}\label{lem1}
Let $\Omega$ be a convex domain and let $f\colon \Omega\to\C$ be holomorphic.
If $\re f'(z)>0$ for all $z\in\Omega$, then $f$ is univalent.
\end{lemma}
The following result can be found in~\cite[p.~112]{Goluzin}.
\begin{lemma}\label{lem2}
Let $a\in\C$, $r>0$ and let $f\colon D(a,r)\to \C$ be univalent. Then
\begin{equation}\label{9i}
\left|\arg\!\left( \frac{f(z)-f(a)}{f'(a)(z-a)}\right)\right|\leq \log\frac{\displaystyle 1+\frac{|z-a|}{r}}{\displaystyle 1-\frac{|z-a|}{r}}
\end{equation}
for $z\in D(a,r)$.
\end{lemma}
The result is stated in~\cite{Goluzin} for the special case that $a=0$, $r=1$, $f(0)=0$ and $f'(0)=1$, 
but the version given above follows directly from this special case.

\begin{proof}[Proof of Proposition~\ref{prop1}]
We recall some arguments of the proof of Theorem~\ref{thm-BE1} in~\cite{BE} and then describe the additional
arguments that have to be made.

As in \cite{BE} we may assume that $L=\R$ and
we use Zalcman's Lemma~\ref{lemma1} as well as Lemma~\ref{lemma2}, applied with $\varphi(t)=(\log t)^2$,
 to obtain a sequence $(z_k)$ tending to $z_0$ and
a sequence $(\rho_k)$ tending to $0$ such that 
\begin{equation}\label{1d9}
R_k
:=\frac 1 {\varrho_k\left(\log \varrho_k\right)^2}\to\infty,
\end{equation}
and the function $g_k$ given by~\eqref{1b}
is defined in the disk $D(0,R_k)$ and satisfies~\eqref{1d2} and $g_k^\#(0)=1$.

As in \cite[Proof of Theorem~1.3]{BE} we find a sequence $(b_k)$ of $1$-points of $g$ such that
\begin{equation}\label{2d}
g_k(z)=\exp\!\left(c_k(z-b_k)+\delta_k(z)\right),
\end{equation}
where (see \cite[(3.4) and (3.5)]{BE})
\begin{equation}\label{3k}
|c_k+2i|\leq\frac{C}{R_k}
\quad\text{or}\quad
|c_k-2i|\leq\frac{C}{R_k}
\end{equation}
with some constant~$C$
and (see \cite[(2.22)]{BE})
\begin{equation}\label{2f}
|\delta_k(z)|\leq 2^{7}\frac{|z-b_k|^2}{R_k}\quad\text{for }|z-b_k|\leq\frac{1}{16}R_k.
\end{equation}
Without loss of generality we may assume that the first alternative holds in~\eqref{3k}.

We put 
\begin{equation}\label{2d1}
h_k(z)=c_k(z-b_k)+\delta_k(z)
\end{equation}
so that $g_k(z)=\exp h_k(z)$.
We will show that $h_k$ is univalent in $D(b_k,2s_k)$ where $s_k=2^{-11}R_k$.
In order to do so we note that for $|z-b_k|\leq 2s_k$ we have
\begin{align}\label{9a}
|\delta_k'(z)|
& = \frac{1}{2\pi}\left|
\int_{|\zeta-b_k|=4s_k} \frac{\delta_k(\zeta)}{(z-\zeta)^2}d\zeta \right|
\leq 4 s_k\frac{1}{(2s_k)^2}\max_{|\zeta-b_k|=4s_k}|\delta_k(\zeta)|.
\end{align}
Since $4s_k=2^{-9}R_k<R_k/16$ we may apply~\eqref{2f} to estimate the maximum 
on the right hand side
and obtain
\begin{align}\label{9a1}
|\delta_k'(z)|
\leq \frac{1}{s_k}2^7 \frac{(4s_k)^2}{R_k}
=1
\quad\text{for}\ |z-b_k|\leq 2 s_k.
\end{align}
Thus, since we assumed that the first alternative holds in~\eqref{3k},
\begin{equation}\label{9b}
\re (ih_k'(z)) =\re (ic_k+i\delta_k'(z))\geq 2-\frac{C}{R_k} -1>0
\end{equation}
for $z\in D(b_k,2s_k)$ if $k$ is sufficiently large. Lemma~\ref{lem1}
 implies that $ih_k$ and hence $h_k$ are univalent in this disk.
Since $h_k(b_k)=0$ and, by~\eqref{2f}, $\delta_k'(b_k)=0$ and thus
$h_k'(b_k)=c_k$, Lemma~\ref{lem2} now yields that if $z\in \overline{D}(b_k,s_k)$, then
\begin{equation}\label{9c}
\left|\arg\!\left( \frac{h_k(z)}{c_k(z-b_k)}\right)\right|
\leq \log 3.
\end{equation}
For $t\in\R$ with $0<|t|\leq s_k$ we thus have 
\begin{equation}\label{9d}
\left|\arg\!\left( \frac{h_k(b_k+it)}{ic_kt}\right)\right| 
\leq \log 3.
\end{equation}
Since we assumed that the first alternative holds in~\eqref{3k},
this implies for large $k$ that 
\begin{equation}\label{9e}
\left|\arg\!\left( h_k(b_k+it)\right)\right| 
\leq \log 3+\arcsin\!\left(\frac{C}{2R_k}\right)< \frac12\pi
\quad\text{for}\ 0<t\leq s_k
\end{equation}
while
\begin{equation}\label{9f}
\left|\arg\!\left( h_k(b_k+it)\right)-\pi\right| < \frac12\pi
\quad\text{for}\  -s_k\leq t<0.
\end{equation}
Hence
\begin{equation}\label{9g}
\re\!\left( h_k(b_k+it)\right)
\begin{cases} 
>0 & \text{if} \  0<t\leq s_k,\\
<0 & \text{if} \   -s_k\leq t<0,
\end{cases}
\end{equation}
so that
\begin{equation}\label{9h}
| g_k(b_k+it)|
=\exp\!\left(\re\!\left( h_k(b_k+it)\right)\right)
\begin{cases} 
>1 & \text{if} \  0<t\leq s_k,\\
<1 & \text{if} \   -s_k\leq t<0.
\end{cases}
\end{equation}
As in~\cite[(3.2), (3.6) and (3.7)]{BE} we put
\[
a_k=z_k+\rho_k b_k,
\quad
u_k=b_k+is_k=b_k+i2^{-11}R_k
\quad\text{and}\quad
\alpha_k=z_k+\rho_k u_k.
\]
By~\eqref{1b} and \eqref{9h} we have $|f_k(z)|>1$ for $z$ in the line segment $(a_k,\alpha_k]$.
Choose $d>r$ such that $\overline{D}(z_0,d)\in D$.
We put $\beta_k=z_k+id$. Then $\beta_k\in D^+$ for large $k$ and as in~\cite{BE} we can use 
Landau's theorem to show that we also have $|f_k(z)|>1$ for $z\in [\alpha_k,\beta_k]$.
Altogether thus $|f_k(z)|>1$ for $z\in (a_k,\beta_k]$ and hence for $z\in M_k\cap D^+\cap \overline{D}(z_0,r)$ and large~$k$.
Analogously, $|f_k(z)|<1$ for $z\in M_k\cap D^-\cap \overline{D}(z_0,r)$ and large~$k$.
\end{proof}
\begin{lemma}\label{lem3}
Let $0<\alpha<\pi$ and $\alpha<\beta<2\pi-\alpha$. Let $u\colon \D\to[-\infty,\infty)$ 
be a subharmonic function which is harmonic in $\D\setminus \{re^{i\beta}\colon 0\leq r<1\}$.
Suppose that $u(z)>0$ for $|\arg z|<\alpha$ while $u(z)\leq 0$ for 
$\alpha\leq |\arg z|\leq \pi$.
Then $\alpha\geq \pi/2$. Moreover, if $\alpha>\pi/2$, then $\beta=\pi$.
In addition, if $u$ is harmonic in $\D\setminus \{0\}$, then $\alpha=\pi/2$.
\end{lemma}
\begin{proof}
Let $\gamma=2\alpha/\pi$ and define $v\colon \{z \in \D\colon \re z\geq 0\}\to [-\infty,\infty)$ by
$v(z)=u(z^\gamma)$. Then $v(z)>0$ for $\re z>0$ while $v(z)\leq 0$ for $\re z=0$.
In fact, $v(z)=0$ for $\re z=0$ by upper semicontinuity.
We have $v=\re f$ for some function $f$ holomorphic in $\{z \in \D\colon \re z> 0\}$.
By the Schwarz reflection principle $f$ extends to a function holomorphic in $\D$.
Hence $f$ has a power series expansion $f(z)=\sum_{k=0}^\infty a_kz^k$ convergent in $\D$.
With $\delta=1/\gamma$ we thus have
\begin{equation}\label{9k}
u(z)=v(z^{1/\gamma})=\re f(z^{\delta})=\re\!\left(\sum_{k=0}^\infty a_kz^{k\delta}
\right)
\end{equation}
for $z\in \D\setminus \{re^{i\beta}\colon 0\leq r<1\}$, meaning that
\begin{equation}\label{9l}
u(re^{i\theta})=\re\!\left(\sum_{k=0}^\infty a_kr^{\delta} e^{ik\delta\theta} \right)
\end{equation}
for $0<r<1$ and $\beta-2\pi<\theta<\beta$.

Since $\re f(z)=v(z)>0$ for $\re z>0$ and $\re f(z)=0$ for $\re z=0$ we find that
$\re a_0=0$ and $a_1>0$.
It follows that 
\begin{equation}\label{9m}
u(re^{i\theta})=a_1r^{\delta} \cos(\delta\theta) +\O(r^{2\delta})
\end{equation}
as $r\to 0$, uniformly for $\beta-2\pi<\theta<\beta$.
We may assume that $\beta\geq \pi$. The condition that 
$u(re^{i\theta})\leq 0$ for $\alpha\leq \theta<\beta$ then implies that 
$\delta\pi\leq \delta\beta \leq 3\pi/2$ so that $\delta\leq 3/2$.
Suppose that $\delta\neq 1$.
Since $u$ is subharmonic and a connected set containing more than one point is non-thin
at every point of its closure \cite[Theorem 3.8.3]{Ransford1995}, 
we have
$u(re^{i\beta})=u(re^{i(\beta-2\pi)})$ and
thus $\cos(\delta\beta)=\cos(\delta(\beta-2\pi))$. This yields that
$\beta=\pi$.
Since $u$ is subharmonic,
we also have
\begin{align}\label{9n}
0=u(0)&\leq \frac{1}{2\pi}\int_{\beta-2\pi}^\beta u(re^{i\theta})d\theta
\\ &
=a_1r^{\delta}\frac{1}{2\pi}\int_{-\pi}^\pi  \cos(\delta\theta)d\theta +\O(r^{2\delta})
\\ &
= a_1r^{\delta}\frac{1}{\delta\pi}\sin (\delta\pi) +\O(r^{2\delta}).
\end{align}
Hence $\sin (\delta\pi)\geq 0$.
Since $\delta\leq 3/2$ and since we assumed that $\delta\neq 1$
this implies that $\delta<1$. Overall thus $\delta\leq 1$ so that  
$\alpha=\gamma\pi/2=\pi/(2\delta)\geq \pi/2$,
and if $\alpha>\pi/2$ so that $\delta<1$, then $\beta=\pi$.
Finally, $u$ can be harmonic only if $\delta=1$, which means that $\alpha=\pi/2$.
\end{proof}
For a bounded domain $G$, a point $z\in G$ and a compact subset $A$ of $\partial G$ let
$\omega(z,A,G)$ denote the harmonic measure of $A$ at a point $z\in G$; see, e.g., \cite[\S 4.3]{Ransford1995}.
It is the solution of the Dirichlet problem for the characteristic function $\chi_A$ of $A$ on 
the boundary of~$G$. Thus
\begin{equation}\label{9j}
\omega(z,A,G)=\sup_u u(z),
\end{equation}
where the supremum is taken over all functions $u$ subharmonic
in $G$ which satisfy 
$\limsup_{z\to\zeta}u(z) \leq \chi_A(\zeta)$ for all $\zeta\in\partial G$.
\begin{lemma}\label{lem4}
Let $G$ and $H$ be bounded domains and let $A\subset\partial G$ and $B\subset \partial H$ be compact.
If $G\subset H$ and 
$A\supset \partial G\cap (H\cup B)$, then
$\omega(z,A,G)\geq \omega(z,B,H)$
for all $z\in G$.
\end{lemma}
\begin{proof}
Let $\zeta\in \partial G\setminus A$. Then $\zeta\in \partial G\setminus (H\cup B)$
and thus $\zeta\in\partial H\setminus B$. Hence
$\lim_{z\to\zeta}\omega(z,B,H)=0$.
We conclude that 
$\limsup_{z\to\zeta}\omega(z,B,H)\leq \chi_A(\zeta)$ for all $\zeta\in\partial G$.
Since $u(z)=\omega(z,B,H)$ is an admissible choice in~\eqref{9j}, the conclusion follows.
\end{proof}

\section{Proof of Theorem~\ref{thm1}}
Without loss of generality we may assume that $L_1$ and $L_{-1}$ are symmetric 
with respect to the real axis and that $L_1$ is in the upper half-plane.
Thus $L_{\pm 1}=\{re^{\pm i\alpha}\colon r\geq 0\}$
for some $\alpha\in (0,\pi)$. We may also assume that 
$L_{0}=\{re^{i\beta}\colon r\geq 0\}$ where $\alpha<\beta<2\pi-\alpha$.
We define
\begin{align}\label{defS}
S&=\{re^{i\theta} \colon 0<r<1,\; |\theta|<\alpha \},\\ 
S^+
& =\{re^{i\theta} \colon 0<r<1,\; \alpha< \theta<\beta \},\\ 
S^-
& =\{re^{i\theta} \colon 0<r<1,\; \beta< \theta<2\pi-\alpha \}.
\end{align}

By Montel's theorem, $\F$ is normal in $\D\setminus (L_1\cup L_0\cup L_{-1})$.
Thus we only have to prove that $\F$ is normal on $\D\cap L_j\setminus \{0\}$ for
$j\in\{0,\pm 1\}$.

First we prove that $\F$ is normal on $\D\cap L_0\setminus \{0\}$.
In order to do so, suppose that $\F$ is not normal at some point
$z_0\in L_0\setminus \{0\}$.
Applying Theorem~\ref{thm-BE1} to the family $\{1-f\colon f\in \F\}$ we see that there exists 
a sequence $(f_k)$ in $\F$ such that 
either $f_k|_{S^+}\to 1$ and $f_k|_{S^-}\to \infty$  or
$f_k|_{S^+}\to \infty$ and $f_k|_{S^-}\to 1$.
Without loss of generality we may assume that the first alternative holds.
If $(f_k)$ is not normal at some $z_1\in L_1\setminus \{0\}$, then -- again by Theorem~\ref{thm-BE1} --
there exists a subsequence of $(f_k)$ which tends to $0$ or to $\infty$ in $S^+$.
This is incompatible with our previous assumption that $f_k|_{S^+}\to 1$.
Hence $(f_k)$ is normal on $\D\cap L_1\setminus \{0\}$.
We conclude that $(f_k)$ is normal in $S^+\cup S\cup L_1\setminus \{0\}$ and hence that
$f_k|_{S^+\cup S\cup L_1\setminus \{0\}}\to 1$.
In particular, $f_k|_{S}\to 1$.
On the other hand, $f_k|_{S^-}\to \infty$.
Hence $(f_k)$ is not normal at any point of $L_{-1}$.
Since $f_k|_{S^-}\to \infty$ we can now deduce from Theorem~\ref{thm-BE1} that 
$f_k|_{S}\to 0$.
This contradicts our previous finding that $f_k|_{S}\to 1$.
Thus $\F$ is normal on $L_0\setminus \{0\}$.
Putting $T=S^+\cup S^-\cup L_0\setminus \{0\}$ we conclude that $\F$ is normal in~$T$.

Suppose now that $\F$ is not normal at some point $z_0\in\D\setminus \{0\}$.
It follows that $z_0\in (L_1\cup L_{-1})\setminus \{0\}$.
Without loss of generality we may assume that $z_0\in L_1\setminus \{0\}$.
Theorem~\ref{thm-BE1} implies that there exists 
a sequence $(f_k)$ in $\F$ such that 
either $f_k|_{S}\to \infty$ and $f_k|_{T}\to 0$  or
$f_k|_{S}\to 0$ and $f_k|_{T}\to \infty$.
In particular we see that the sequence $(f_k)$ is not normal at any point of $L_1\cup L_{-1}$.
We begin by considering the case that the first of the two above possibilities holds; that is,
$f_k|_{S}\to \infty$ and $f_k|_{T}\to 0$.

We define $u_k\colon \D\to [-\infty,\infty)$, 
\begin{equation}\label{8a}
u_k(z)=\frac{\log|f_k(z)|}{\log|f_k(\frac12)|}.
\end{equation}
We will prove that the sequence $(u_k)$ is locally bounded in~$\D$.
Once this is known, we can deduce (see, for example,
\cite[Theorems 4.1.8, 4.1.9]{Hor} or
\cite[Theorems 3.2.12, 3.2.13]{Hor2})
that some subsequence of $(u_k)$ converges to a limit function $u$ which is subharmonic in $\D$ and
harmonic in $\D\setminus L_0$.
Moreover, $u(z)>0$ for $z\in S$ while $u(z)\leq 0$ for $z\in \D\setminus S$.

Lemma~\ref{lem3} now implies that $\alpha\geq \pi/2$ and that $\beta=\pi$ if $\alpha>\pi/2$.
The conclusion then follows since if $\alpha=\pi/2$, then 
$\angle(L_{-1},L_1)=2\alpha =\pi$, while if 
$\alpha>\pi/2$ and thus $\beta=\pi$, then
$\angle(L_{0},L_1)=\beta-\alpha =\pi-\alpha<\pi/2$ and
$\angle(L_{0},L_{-1})=2\pi-\alpha -\beta=\pi-\alpha=\angle(L_{0},L_1)$.

In order to prove that $(u_k)$ is locally bounded,
let $0<\varepsilon<1/8$. Proposition~\ref{prop1} yields that, for sufficiently large~$k$, there exist
simple closed curves $\Gamma_k$ in $\{z\colon 1-\varepsilon/2<|z|<1-\varepsilon/4\}$
and $\gamma_k$ in $\{z\colon \varepsilon/2<|z|<\varepsilon\}$ such that
$|f_k(z)|>1$ for $z\in (\Gamma_k\cup\gamma_k) \cap S$ 
while 
$|f_k(z)|<1$ for $z\in (\Gamma_k\cup\gamma_k) \cap T$.
Moreover, both $\Gamma_k$ and $\gamma_k$ 
surround $0$ and they intersect $L_1$ and $L_{-1}$ only once, at $1$-points of $f_k$.
In fact, these curves can be constructed by taking small segments orthogonal to 
$L_1$ and $L_{-1}$, and connecting the endpoints of these segments within the intersection
of $S$ and $T$ with the corresponding annuli.

Let $D_k$ be the domain between $\gamma_k$ and $\Gamma_k$ and let
$X_k$ be the set of all $z\in \overline{D_k}$ for which $|f_k(z)|=1$.
Then both $X_k\cap \Gamma_k$ and $X_k\cap \gamma_k$ consist of two $1$-points of $f_k$.
Let $U_k$ be the component of $D_k\setminus X_k$ which contains $\frac12$.
Next, for large $k$ we have $|f_k(z)|<1$ for $z\in L_0$ with $\varepsilon/2\leq |z|\leq 1-\varepsilon/4$ and
hence in particular for $z\in L_0\cap D_k$.
Let $V_k$ be the component  of $D_k\setminus X_k$ which contains $L_0\cap D_k$.
Then, for large $k$, we have $|f_k(z)|>1$ for $z\in U_k$ while $|f_k(z)|<1$ for $z\in V_k$. 

We claim that $D_k\setminus X_k=U_k\cup V_k$.
Indeed, let $W$ be a component of $D_k\setminus X_k$ which is different from $U_k$ and $V_k$.
Since $(\Gamma_k\cup\gamma_k)\cap S\subset \partial U_k$ 
and $(\Gamma_k\cup\gamma_k)\cap T\subset \partial V_k$  we have $\partial W\subset X_k$ for large~$k$.
By the maximum principle, we thus have $|f_k(z)|<1$ for $z\in W$. The minimum principle now yields
that $W$ contains a zero of $f_k$. Hence $W$ and thus $\partial W$ intersect $L_0\cap \overline{D_k}$, 
which is a contradiction for large~$k$, since $\partial W\subset X_k$ and thus $|f_k(z)|=1$ for 
$z\in\partial W$, but $f_k|_{L_0\cap \overline{D_k}}\to 0$.
Thus $D_k\setminus X_k=U_k\cup V_k$ as claimed.
We also conclude that $X_k$ consists of two analytic curves $\sigma_{1,k}$ and $\sigma_{-1,k}$, which are 
close to the rays $L_1$ and $L_{-1}$. 

We now prove that $(u_k)$ is bounded in $\overline{D}(0,1-\varepsilon)$. 
In order to do so we choose $c_k\in \partial D(0,1-\varepsilon)$ such that
\begin{equation}\label{8a1}
u_k(c_k)=\max_{|z|=1-\varepsilon} u_k(z) .
\end{equation}
Clearly, $c_k\in U_k$ for large $k$. 
For $j\in\{1,2,3,4\}$, we put $r_j=1-\varepsilon j/4$. Thus $|c_k|=1-\varepsilon=r_4$.
Similar to the curve $\Gamma_k$ in $\{z\colon r_2<|z|<r_1\}$ there exists a closed curve
$\Gamma_k'$ in $\{z\colon r_4<|z|<r_3\}$ which surrounds $0$
such that $|f_k(z)|>1$ for $z\in \Gamma_k'\cap S$ while 
$|f_k(z)|<1$ for $z\in \Gamma_k'\cap T$.
Thus $\Gamma_k'\cap S\subset U_k$ and $\Gamma_k'\cap T\subset V_k$.
 
By the maximum principle, there exists a curve $\xi_k$ in $U_k$
which connects $c_k$ with $\partial \D$ and on which $u_k$ is bigger than $u_k(c_k)$.
Let $\tau_k$ be a part of $\xi_k$ which connects $\partial D(0,r_3)$ with $\partial D(0,r_2)$ and,
except for its endpoints, is contained in $\{z\colon r_3<|z|<r_2\}$; see Figure~\ref{fig1}.
\begin{figure}[!htb]
\centering
\begin{tikzpicture}[scale=5.,>=latex](-0.1,-1.1)(-1.1,1.1)
\clip (-0.1,-1.1) rectangle (2.6,1.1);
\draw[dash dot dot,->] (-0.1,0) -- (1.1,0);
\draw[dash dot dot,->] (0.,-1.1) -- (0.,1.1);
\draw[black] (0,0) circle (1.0);
\draw[black,dotted] (0,0) circle (0.8);
\draw[black,dotted] (0,0) circle (0.7);
\draw[black,dotted] (0,0) circle (0.6);
\draw[-] (0.1708,-0.6375) -- (0.4666,-0.4666);
\draw[-] (0.1708,0.6375) -- (0.4666,0.4666);
\draw [black,domain=-45:45,samples=100] plot ({0.66*cos(\x)}, {0.66*sin(\x)});
\draw [black,domain=-75:-100,samples=100] plot ({0.66*cos(\x)}, {0.66*sin(\x)});
\draw [black,domain=75:100,samples=100] plot ({0.66*cos(\x)}, {0.66*sin(\x)});
        \node[below right] at (0.0,-0.65) {$\Gamma_{k}'$};
\draw[black, dotted, rotate around={70:(0,0)}] (0,0) -- (1.0,0);
\draw[black, dotted, rotate around={60:(0,0)}] (0,0) -- (1.0,0);
\draw[black, dotted, rotate around={50:(0,0)}] (0,0) -- (1.0,0);
\draw[black, dotted, rotate around={-70:(0,0)}] (0,0) -- (1.0,0);
\draw[black, dotted, rotate around={-60:(0,0)}] (0,0) -- (1.0,0);
\draw[black, dotted, rotate around={-50:(0,0)}] (0,0) -- (1.0,0);
\draw[black,thick, rotate around={50:(0,0)}] (0.25,0) -- (0.8,0);
\draw[black,thick, rotate around={-50:(0,0)}] (0.25,0) -- (0.7,0);
\draw [black,thick,domain=-50:50,samples=100] plot ({0.25*cos(\x)}, {0.25*sin(\x)});
\draw [black,thick,domain=-50:-62,samples=100] plot ({0.7*cos(\x)}, {0.7*sin(\x)});
\draw [black,thick,domain=-42:50,samples=100] plot ({0.8*cos(\x)}, {0.8*sin(\x)});
\draw [black, ultra thick] plot [smooth, tension=0.5] coordinates { (.32,-0.6226) (0.4,-0.63) (0.53,-0.53) (0.6,-0.5291)};
\filldraw[black] (0.25,0) circle (0.01);
\filldraw[black] (0.5,0) circle (0.01);
\filldraw[black] (0.6,0) circle (0.01);
\filldraw[black] (0.7,0) circle (0.01);
\filldraw[black] (0.8,0) circle (0.01);
\filldraw[black] (1.0,0) circle (0.01);
\node[below right] at (0.23,0) {\small $\tfrac14$};
\node[below right] at (0.46,0) {\small $\tfrac12$};
\node[below right] at (0.56,0) {\small $r_4$};
\node[below right] at (0.67,0) {\small $r_3$};
\node[below right] at (0.78,0) {\small $r_2$};
\node[below right] at (0.98,0) {\small $1$};
\node[below right] at (0.4,-0.6) {\small $\tau_k$};
\node[below right] at (0.75,-0.37) {\small $\xi_k$};
\filldraw[black] (0.5,0.866) circle (0.01);
        \node[above right] at (0.5,0.866) {$e^{i\alpha}$};
\filldraw[black] (0.6427,0.766) circle (0.01);
        \node[above right] at (0.6427,0.766) {$e^{i(\alpha-\varepsilon)}$};
\filldraw[black] (0.3420,0.939) circle (0.01);
        \node[above right] at (0.3420,0.939) {$e^{i(\alpha+\varepsilon)}$};
\filldraw[black] (0.5,-0.866) circle (0.01);
        \node[right] at (0.5,-0.876) {$e^{-i\alpha}$};
\filldraw[black] (0.6427,-0.766) circle (0.01);
        \node[right] at (0.6427,-0.776) {$e^{-i(\alpha-\varepsilon)}$};
\filldraw[black] (0.3420,-0.939) circle (0.01);
        \node[right] at (0.3420,-0.949) {$e^{-i(\alpha+\varepsilon)}$};
\filldraw[black] (0.32,-0.6226) circle (0.01);
        \node[below right] at (0.56,-0.52) {$e_{k,2}$};
\filldraw[black] (0.6,-0.5291) circle (0.01);
        \node[below left] at (0.38,-0.61) {$e_{k,3}$};
\filldraw[black] (0.34,-0.4943) circle (0.01);
        \node[above left] at (0.35,-0.51) {$c_{k}$};
\draw [black] plot [smooth, tension=0.5] coordinates { (.32,-0.6226) (0.29,-0.6) (0.35,-0.55) (0.34,-0.4943)};
\draw [black] plot [smooth, tension=0.5] coordinates { (0.85,-0.5267) (0.75,-0.45) (0.7,-0.4) (0.63,-0.45) (0.65,-0.53) (0.6,-0.5291)};
\draw[dash dot dot,->] (1.4,0) -- (2.6,0);
\draw[dash dot dot,->] (1.5,-1.1) -- (1.5,1.1);
\draw [black,domain=-96:96,samples=100] plot ({1.5+cos(\x)}, {sin(\x)});
\draw [black,dotted,domain=-97.5:97.5,samples=100] plot ({1.5+0.8*cos(\x)}, {0.8*sin(\x)});
\draw [black,dotted,domain=-98.5:98.5,samples=100] plot ({1.5+0.7*cos(\x)}, {0.7*sin(\x)});
\draw[black, dotted, rotate around={70:(1.5,0)}] (1.5,0) -- (2.5,0);
\draw[black, dotted, rotate around={60:(1.5,0)}] (1.5,0) -- (2.5,0);
\draw[black, dotted, rotate around={50:(1.5,0)}] (1.5,0) -- (2.5,0);
\draw[black, dotted, rotate around={-70:(1.5,0)}] (1.5,0) -- (2.5,0);
\draw[black, dotted, rotate around={-60:(1.5,0)}] (1.5,0) -- (2.5,0);
\draw[black, dotted, rotate around={-50:(1.5,0)}] (1.5,0) -- (2.5,0);
\draw[black, ultra thick, rotate around={-70:(1.5,0)}] (2.2,0) -- (2.3,0);
\draw[black,thick, rotate around={-50:(1.5,0)}] (1.75,0) -- (2.2,0);
\draw[black,thick, rotate around={50:(1.5,0)}] (1.75,0) -- (2.3,0);
\draw[black,dotted] (0,0) circle (0.8);
\draw [black,thick,domain=-50:50,samples=100] plot ({1.5+0.25*cos(\x)}, {0.25*sin(\x)});
\draw [black,thick,domain=-70:0,samples=100] plot ({1.5+0.7*cos(\x)}, {0.7*sin(\x)});
\draw [black,thick,domain=-70:50,samples=100] plot ({1.5+0.8*cos(\x)}, {0.8*sin(\x)});
\filldraw[black] (1.75,0) circle (0.01);
\filldraw[black] (2.0,0) circle (0.01);
\filldraw[black] (2.2,0) circle (0.01);
\filldraw[black] (2.3,0) circle (0.01);
\filldraw[black] (2.5,0) circle (0.01);
\node[below right] at (1.73,0) {\small $\tfrac14$};
\node[below right] at (1.96,0) {\small $\tfrac12$};
\node[below right] at (2.18,0) {\small $r_3$};
\node[below right] at (2.28,0) {\small $r_2$};
\node[below right] at (2.48,0) {\small $1$};
\node[below right] at (1.64,-0.66) {\small $K$};
\filldraw[black] (2.0,0.866) circle (0.01);
        \node[above right] at (2.0,0.866) {$e^{i\alpha}$};
\filldraw[black] (2.1427,0.766) circle (0.01);
        \node[above right] at (2.1427,0.766) {$e^{i(\alpha-\varepsilon)}$};
\filldraw[black] (1.8420,0.939) circle (0.01);
        \node[above right] at (1.8420,0.939) {$e^{i(\alpha+\varepsilon)}$};
\filldraw[black] (2.0,-0.866) circle (0.01);
        \node[right] at (2.0,-0.876) {$e^{-i\alpha}$};
\filldraw[black] (2.1427,-0.766) circle (0.01);
        \node[right] at (2.1427,-0.776) {$e^{-i(\alpha-\varepsilon)}$};
\filldraw[black] (1.8420,-0.939) circle (0.01);
        \node[right] at (1.8420,-0.949) {$e^{-i(\alpha+\varepsilon)}$};
\end{tikzpicture}
\caption{The curves $\xi_k$, $\tau_k$ and $\Gamma_k'$ and the domains $G_k$ (left) and $H$ (right).}
\label{fig1}
\end{figure}
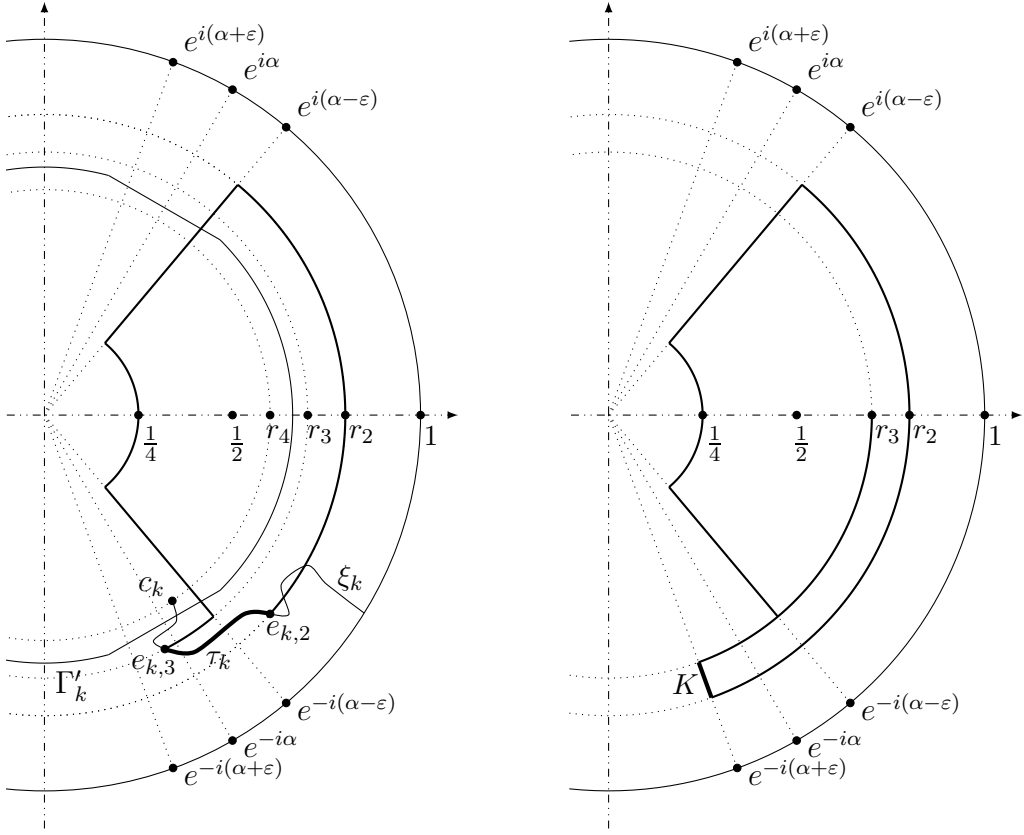
Then $|u_k(z)|\geq u_k(c_k)$ for $z\in \tau_k$.
Let $e_{k,j}$ be the endpoint of $\tau_k$ on $\partial D(0,r_j)$, for $j\in\{2,3\}$.
Without loss of generality we may assume that the distance of $e_{k,3}$ to $L_{-1}$ is less than
or equal to
the distance to $L_{1}$, which means that $\im e_{k,3}\leq 0$.

We define a domain $G_k$ as follows; cf.\ Figure~\ref{fig1}.
If $\tau_k$ does not intersect the segment
$\{re^{i(\alpha-\varepsilon)}\colon r_3 \leq r\leq r_2\}$, let $G_k$ be 
the domain bounded by the segments
$\{re^{-i(\alpha-\varepsilon)}\colon \frac14\leq r\leq r_3\}$ and
$\{re^{i(\alpha-\varepsilon)}\colon \frac14\leq r\leq r_2\}$,
the arc $\{\frac14 e^{i\theta}\colon |\theta|\leq \alpha-\varepsilon\}$,
the arc of $\partial D(0,r_3)$ that connects $e_{k,3}$ and $r_3 e^{-i(\alpha-\varepsilon)}$ in
$\{r_3e^{i\theta}\colon |\theta|\leq \alpha+\varepsilon\}$, 
the arc of $\partial D(0,r_2)$ that connects $e_{k,2}$ and $r_2 e^{i(\alpha-\varepsilon)}$ in 
$\{r_2e^{i\theta}\colon |\theta|\leq \alpha+\varepsilon\}$, 
and the curve~$\tau_k$.

If $\tau_k$ intersects the segment
$\{re^{i(\alpha-\varepsilon)}\colon r_3\leq r\leq r_2 \}$, let 
$d_k$ denote the first point of intersection so that the part $\tau_k'$ of
$\tau_k$ which is between $e_{k,3}$ and $d_k$ is contained in 
$\{re^{i\theta}\colon r_3< r< r_2, -\alpha-\varepsilon<\theta<\alpha-\varepsilon\}$.
We then define $G_k$ as the domain bounded by the 
the curve $\tau_k'$, the  segment $\{re^{i(\alpha-\varepsilon)}\colon \frac14\leq r\leq |d_k|\}$ and
-- as before --
the arc $\{\frac14 e^{i\theta}\colon |\theta|\leq \alpha-\varepsilon\}$,
the segment $\{re^{-i(\alpha-\varepsilon)}\colon \frac14\leq r\leq r_3\}$ and
the arc of $\partial D(0,r_3)$ that connects $e_{k,3}$ and $r_3 e^{-i(\alpha-\varepsilon)}$ in 
$\{r_3e^{i\theta}\colon |\theta|\leq \alpha+\varepsilon\}$.

We claim that $G_k\subset U_k$ for large~$k$.
In order to prove this it suffices to prove that $\partial G_k\subset U_k$.
We restrict to the case that $\tau_k$ does not intersect the segment
$\{re^{i(\alpha-\varepsilon)}\colon r_3 \leq r\leq r_2\}$, since the other case
is similar.
First we note that the segments
$\{re^{-i(\alpha-\varepsilon)}\colon \frac14\leq r\leq r_3\}$ and
$\{re^{i(\alpha-\varepsilon)}\colon \frac14\leq r\leq r_2\}$ as well as
the arc $\{\frac14 e^{i\theta}\colon |\theta|\leq \alpha-\varepsilon\}$
are clearly in $U_k$ for large~$k$, since $f_k|_S\to\infty$ as $k\to\infty$.
Since $\xi_k$ is in $U_k$ and $\tau_k$ is a subcurve of $\xi_k$, the curve $\tau_k$ 
is also in $U_k$.

It remains to show that the arc of $\partial D(0,r_3)$ that connects $e_{k,3}$ and
$r_3 e^{-i(\alpha-\varepsilon)}$ is in $U_k$. If this is not the case, then this arc
must intersect $\partial U_k$ and thus must intersect the curve 
$\sigma_{-1,k}$, which constitutes the part of $\partial U_k$ that is near $L_{-1}$.
Since $\xi_k$ is in $U_k$ this means that $\sigma_{-1,k}$ must also intersect $\Gamma_k'$,
at a point between the intersections of $\Gamma_k'$ with $\xi_k$ and with the positive 
real axis. But this part of $\Gamma_k'$ is in $U_k$, since 
$\Gamma_k'\cap S\subset U_k$ and $\Gamma_k'\cap T\subset V_k$.
Hence $\sigma_{-1,k}$ does not intersect the arc connecting $e_{k,3}$ and $r_3 e^{-i(\alpha-\varepsilon)}$
and thus this arc is in $U_k$. Similarly, we see that the arc 
of $\partial D(0,r_2)$ that connects $e_{k,2}$ and $r_2 e^{i(\alpha-\varepsilon)}$ is in $U_k$.
Altogether thus $G_k\subset U_k$ for large~$k$.

Let 
\begin{align}\label{defH}
H=
&\{re^{i\theta}\colon \tfrac14< r< r_3, |\theta|< \alpha-\varepsilon\}
\cup\{ r_3e^{i\theta}\colon 0<\theta <\alpha-\varepsilon\}
\\ &
\cup \{re^{i\theta}\colon r_3< r< r_2,
-\alpha-\varepsilon< \theta< \alpha-\varepsilon\}
\end{align}
and let $K=\{re^{i(-\alpha-\varepsilon)}\colon r_3\leq r\leq r_2\}\subset \partial H$; see Figure~\ref{fig1}.
Then $G_k\subset H$.

It follows from Lemma~\ref{lem4} and the configuration of the domains $G_k$ and $H$ that 
$\omega(z,K,H)\leq \omega(z,\tau_k,G_k)$ for $z\in G_k$.
In particular,
$\omega(\frac12,K,H)\leq \omega(\frac12,\tau_k,G_k)$.
On the other hand, since $G_k\subset U_k$ and thus $u_k(z)\geq 0$ for $z\in\partial G_k$ while 
$u_k(z)\geq u_k(c_k)$ for $z\in\tau_k$ it follows that
$u_k(z)\geq u_k(c_k)\omega(z,\tau_k,G_k)$ for $z\in G_k$.
Altogether we thus have 
\begin{equation}\label{8b}
1=u_k(\tfrac12)\geq   u_k(c_k)\omega(\tfrac12,\tau_k,G_k)\geq  u_k(c_k)\omega(\tfrac12,K,H).
\end{equation}
It follows that
\begin{equation}\label{8a2}
\max_{|z|=1-\varepsilon} u_k(z) = u_k(c_k) \leq \frac{1}{\omega(\tfrac12,K,H)}
\end{equation}
so that $(u_k)$ is bounded in $\overline{D}(0,1-\varepsilon)$. Since $\varepsilon>0$ can be taken
arbitrarily small, we conclude that $(u_k)$ is locally bounded in $\D$.
This completes the proof in the case that
$f_k|_{S}\to \infty$ and $f_k|_{T}\to 0$.

It remains to consider the case that
$f_k|_{S}\to 0$ and $f_k|_{T}\to \infty$.
Since $L_0\setminus\{0\}\subset T$ we conclude that if $\varepsilon>0$, then, for large~$k$,
the function $f_k$ has no zeros in $\{z\colon \varepsilon<|z|<1-\varepsilon\}$. Thus $u_k$ is harmonic there.
As before we see that the sequence $(u_k)$ is locally 
bounded so that some subsequence of it converges to a limit $u$ which is subharmonic in $\D$.
But now $u$ is actually harmonic in $\D\setminus \{0\}$.
The conclusion follows again from  Lemma~\ref{lem3} which yields that 
$\alpha=\pi/2$ and hence $\angle(L_{-1},L_1)=2\alpha =\pi$. \qed

\noindent
\emph{Walter Bergweiler}\\
{\sc Mathematisches Seminar\\
Christian-Albrechts-Universit\"at zu Kiel\\
Ludewig-Meyn-Str.\ 4\\
24098 Kiel\\
Germany}\\
{\tt Email: bergweiler@math.uni-kiel.de}

\medskip

\noindent
\emph{Alexandre Eremenko}\\
{\sc Department of Mathematics\\
Purdue University\\
West Lafayette, IN 47907\\
USA}\\
{\tt Email: eremenko@math.purdue.edu}
\end{document}